\documentclass[12pt]{amsart}
\usepackage{amssymb,amsmath,amsthm,hyperref,mathrsfs,multirow,xcolor}
\usepackage{tikz}
\usetikzlibrary{arrows,matrix,positioning}
\newtheorem{theorem}{Theorem}
\newtheorem{lemma}{Lemma}
\newtheorem{cor}[theorem]{Corollary}

\newtheorem{prop}[theorem]{Proposition}
\newtheorem*{theoremNoNum}{Theorem}
\newtheorem*{conjNoNum}{Conjecture}

\theoremstyle{definition}
\newtheorem{definition}{Definition}

\theoremstyle{remark}

\allowdisplaybreaks

\begin{document}
\newcommand\mylabel[1]{\label{#1}}
\newcommand{\beqs}{\begin{equation*}}
\newcommand{\eeqs}{\end{equation*}}
\newcommand{\beq}{\begin{equation}}
\newcommand{\eeq}{\end{equation}}
\newcommand\eqn[1]{(\ref{eq:#1})}
\newcommand\exam[1]{\ref{exam:#1}}
\newcommand\thm[1]{\ref{thm:#1}}
\newcommand\lem[1]{\ref{lem:#1}}
\newcommand\propo[1]{\ref{propo:#1}}
\newcommand\corol[1]{\ref{cor:#1}}
\newcommand\sect[1]{\ref{sec:#1}}
\newcommand\subsect[1]{\ref{subsec:#1}}

\newcommand{\Z}{\mathbb Z}
\newcommand{\N}{\mathbb N}
\newcommand{\R}{\mathbb R}
\newcommand{\C}{\mathbb C}
\newcommand{\Q}{\mathbb Q}
\newcommand{\cA}{\mathcal A}
\newcommand{\cC}{\mathcal C}
\newcommand{\even}{\mathrm{even}}
\newcommand{\odd}{\mathrm{odd}}

\title[Inverse of a matrix related to double zeta values of odd weight]
{Inverse of a matrix related to double zeta values\\ of odd weight}

\author{Ding Ma}
\address{Department of Mathematics, University of Arizona, Tucson, AZ. 85721}
\email{martin@math.arizona.edu}

\date{\today}

\begin{abstract}
In this paper, we give a proof of a conjecture made by Zagier in \cite{Zag1} about the inverse of some matrix related to double zeta values of parity $(\even,\odd)$. As a result, we obtain a family of Bernoulli number identities. We further generalize this family to a more general setting involving binomial coefficients of negative arguments.
\end{abstract}

\maketitle

\section{Introduction and main result}\label{sc1}

Multiple zeta values are real numbers, first defined by Euler in $1700$s, that have been much studied in recent years because of their many surprising properties. They appear in many places in both mathematics and mathematical physics, from periods of mixed Tate motives to values of Feynman integrals. There are many conjectures concerning the values of these numbers. 

We first give the definition and formulations of some open problems concerning multiple zeta values (MZVs) that we will prove here. For positive integers $k_1,\ldots,k_n$ with $k_n\geq 2$, we define the MZV $\zeta(k_1,\ldots,k_n)$ as the iterated multiple sum
\begin{equation}
\zeta(k_1,\ldots,k_n)=\sum_{0<m_1<\cdots<m_n}\frac{1}{m_1^{k_1}\cdots m_n^{k_n}}.
\end{equation}
The constant $n$ is called the depth of this MZV, and $k=k_1+\cdots+k_n$ is its weight. It was already found by Euler (explicitly for $k$ up to 13) that all double zeta values (MZVs of depth $2$) of odd weight are rational linear combinations of products of ``Riemann'' zeta values. In \cite{Zag1}, Zagier gave the following explicit expression for double zeta values of weight $k=2K+1$ and parity $(\even,\odd)$
\begin{equation}\label{ep}
\zeta(2r,k-2r)\equiv \sum_{s=1}^{K-1}\left[{2K-2s\choose 2r-1}+{2K-2s\choose 2K-2r}\right]\zeta(2s)\zeta(k-2s)\quad(1\leq r\leq K-1),
\end{equation}
where the congruence is modulo $\Q\zeta(k)$. He further defined a matrix $\cA=\cA_K$ to be the $(K-1)\times(K-1)$ matrix whose $(r,s)$-entry is the expression in square brackets in (\ref{ep}), and made the following conjecture about the entries of $\cA^{-1}$.

\begin{conjNoNum} For any odd integer $k=2K+1\geq 5$ and the matrix $\cA=\cA_K$ defined above, we have
\begin{eqnarray*}
\big(\cA^{-1}\big)_{s,r}&\stackrel{?}{=}&\frac{-2}{2s-1}\sum_{n=0}^{k-2s}{k-2r-1\choose k-2s-n}{n+2s-2\choose n}B_n\\
&\stackrel{?}{=}&\frac{2}{2s-1}\sum_{n=0}^{k-2s}{2r-1\choose k-2s-n}{n+2s-2\choose n}B_n\qquad(1\leq s,r\leq K-1),
\end{eqnarray*}
where $B_n$ denotes the $n$th Bernoulli number. 
\end{conjNoNum}

In order to match the notation in an ongoing project, I will consider the matrix $\cC=\cA^T$, and prove the following the result in this paper.
\begin{theorem}\label{thm:Cinv}
For any odd integer $k=2K+1\geq 5$ and the matrix $\cC=\cA_K^T$ defined above, we have
\begin{eqnarray} 
\big(\cC^{-1}\big)_{r,s}&=&\frac{-2}{2s-1}\sum_{n=0}^{k-2s}{k-2r-1\choose k-2s-n}{n+2s-2\choose n}B_n \label{eq:Cinv1}\\
&=&\frac{2}{2s-1}\sum_{n=0}^{k-2s}{2r-1\choose k-2s-n}{n+2s-2\choose n}B_n\qquad(1\leq s,r\leq K-1) \label{eq:Cinv2},
\end{eqnarray}
\end{theorem}

In particular, the first and the last rows of $\cC^{-1}$ consist of simple multiples of Bernoulli numbers as the next corollary says.

\begin{cor}\label{cor:r1}
For any odd integer $k=2K+1\geq 5$ and any $s$ satisfying $1\leq s\leq K-2$, we have
$$\big(\cC^{-1}\big)_{1,s}=-\frac{1}{2}\big(\cC^{-1}\big)_{K-1,s}=2{2K-2\choose 2s-1}\frac{B_{2K-2s}}{2K-2s}.$$
\end{cor}	

By using $\cC^{-1}$, we can explicitly express the products $\zeta(2s)\zeta(k-2s)$, $1 \leq s \leq K-1$ in terms of double zeta values $\zeta(2r,k-2r)$, $1\leq r\leq K-1$.

\begin{cor}\label{cor:r2}
For odd $k = 2K+1 \geq 5$, modulo $\Q\zeta(k)$, the products $\zeta(2s)\zeta(k-2s)$, $1 \leq s \leq K-1$, can be expressed in terms of double zeta values $\zeta(2r,k-2r)$, $1\leq r\leq K-1$ as follows
\begin{equation}
\zeta(2s)\zeta(k-2s)\equiv\sum_{r=1}^{K-1}\big(\cC^{-1}\big)_{r,s}\zeta(2r,k-2r),
\end{equation}
where $\big(\cC^{-1}\big)_{r,s}$ is given by either \eqn{Cinv1} or \eqn{Cinv2}.
\end{cor}
 
In Section \ref{sc2}, we will define two $2$-variable polynomials corresponding to the two expressions of rows of $\cC^{-1}$, and prove Theorem \thm{Cinv}, Corollary \corol{r1} by using Theorem \thm{FFR}. In Section \ref{sc3}, we will state four lemmas, and use them to prove Theorem \thm{FFR}. In Section \ref{sc4}, we will prove all the lemmas stated in Section \ref{sc3}. Finally, in Section \ref{sc5}, we will extend some of the definitions and results stated before, and give a family of Bernoulli number identities.

\section{Proof of Theorem \thm{Cinv}}\label{sc2}

In this section, we will define two $2$-variable polynomials, and state a result about one of them. Later, in Section \ref{sc3}, we will see that the result also holds for the other one by proving that they are identical to each other.

\begin{definition}\label{dFG}
Let $k=2K+1\geq 5$ be an odd integer, for any $0\leq i\leq k-2$, we define the following two polynomials
\begin{eqnarray}
F_k^{(i)}(x,y)&:=&\sum_{s=1}^{K-1}\left(\frac{-2}{2s-1}\sum^{k-2s}_{n=0}{i\choose k-2s-n}{n+2s-2\choose n}B_n\right)x^{2K-2s}y^{2s-1};\\
G_k^{(i)}(x,y)&:=&\sum_{s=1}^{K-1}\left(\frac{2}{2s-1}\sum^{k-2s}_{n=0}{k-2-i\choose k-2s-n}{n+2s-2\choose n}B_n\right)x^{2K-2s}y^{2s-1}.
\end{eqnarray}
\end{definition}

Those two polynomials correspond to the two expression \eqn{Cinv1}, \eqn{Cinv2} of $\cC^{-1}$, and they are closely related to the following polynomial.

\begin{definition} Let $k=2K+1\geq 5$ be an odd integer, for any $0\leq i\leq k-2$, we define the following polynomial
\begin{equation}
R_k^{(i)}(x,y):=-\frac{k-2-2i}{k-2}x^{k-2}+(-1)^{i}x^{k-2-i}y^{i}+(-1)^{i}x^{i}y^{k-2-i}-\frac{k-2-2i}{k-2}y^{k-2}
\end{equation}
\end{definition}

The connection between $F_k^{(i)}(x,y)$, $G_k^{(i)}(x,y)$  and $R_k^{(i)}(x,y)$ can be stated as the following theorem, whose proof will be postponed to Section $3$.
\begin{theorem}\label{thm:FFR}
Let $k=2K+1\geq 5$ be an odd integer, for any $0\leq i\leq k-2$, we have
\begin{eqnarray}
F_k^{(i)}(x+y,x)+F_k^{(i)}(x+y,y)=R_k^{(i)}(x,y) \label{eq:FFR}\\
G_k^{(i)}(x+y,x)+G_k^{(i)}(x+y,y)=R_k^{(i)}(x,y) \label{eq:GGR}
\end{eqnarray}
\end{theorem}

Assuming Theorem \thm{FFR}, we are able to prove Theorem \thm{Cinv} now.
\begin{proof}[Proof of Theorem \thm{Cinv}]
Let $k = 2K + 1 \geq 5$ be an odd integer. According to the definition of the matrix $\cC$, we have
\begin{equation}
\cC_{s,r}={2K-2s\choose 2r-1}+{2K-2s\choose 2K-2r}.
\end{equation}
For the vector $v=(v_1,v_2,\ldots,v_{K-1})$, define its associated polynomial to be $V(x,y)=\sum_{s=1}^{K-1}v_s x^{2K-2s} y^{2s-1}$. Assume that $v\cC=w:=(w_1,w_2,\ldots,w_{K-1})$, then
\begin{eqnarray*}
W(x,y)&:=&\sum_{j=1}^{K-1}w_j x^{2j-1}y^{2K-2j}\\
&=&\textrm{$x^{2j-1}y^{2K-2j}$-terms in $V(x+y,x)+V(x+y,y)$ with $1\leq j\leq K-1$}.
\end{eqnarray*}
	
According to Theorem \thm{FFR}, for $1\leq i\leq K-1$, we have 
\begin{eqnarray}
&&\textrm{$x^{2j-1}y^{2K-2j}$-terms in $F_k^{(2i-1)}(x+y,x)+F_k^{(2i-1)}(x+y,y)$ with $1\leq j\leq K-1$}\\
&=&\textrm{$x^{2j-1}y^{2K-2j}$-terms in $R_k^{(2i-1)}(x,y)$ with $1\leq j\leq K-1$}\\
&=&x^{2i-1}y^{2K-2i}.
\end{eqnarray}
Therefore, for the matrix with the $i$th row coming from the associated vector of $F_k^{(2i-1)}(x,y)$, its product with $\cC$ gives us the identity matrix, i.e., we have
\begin{equation*} 
\big(\cC^{-1}\big)_{r,s}=\frac{2}{2s-1}\sum_{n=0}^{k-2s}{2r-1\choose k-2s-n}{n+2s-2\choose n}B_n\qquad(1\leq s,r\leq K-1).
\end{equation*}
By considering $G_k$ instead of $F_k$, we get another expression
\begin{equation*} 
\big(\cC^{-1}\big)_{r,s}=\frac{-2}{2s-1}\sum_{n=0}^{k-2s}{k-2r-1\choose k-2s-n}{n+2s-2\choose n}B_n\qquad(1\leq s,r\leq K-1).
\end{equation*}
Hence we have proven the statement.
\end{proof}

As a corollary, the first and the last rows of $\cC^{-1}$ can be reduced to simple multiples of Bernoulli numbers.

\begin{proof}[Proof of Corollary \corol{r1}]
Let $k = 2K + 1 \geq 5$ be an odd integer. When $r=1$ and $1\leq s\leq K-2$, we have
\begin{eqnarray*}
\big(\cC^{-1}\big)_{1,s}&\stackrel{\eqn{Cinv2}}{=}&\frac{2}{2s-1}\sum_{n=0}^{k-2s}{1\choose k-2s-n}{n+2s-2\choose n}B_n\\
&=&\frac{2}{2s-1}{k-2s-1+2s-2\choose k-2s-1}B_{k-2s-1}\\
&=&\frac{2}{2s-1}{2K-2\choose 2s-2}B_{2K-2s}\\
&=&2{2K-2\choose 2s-1}\frac{B_{2K-2s}}{2K-2s}.
\end{eqnarray*}
When $r=K-1$ and $1\leq s\leq K-2$, we have
\begin{eqnarray*}
\big(\cC^{-1}\big)_{K-1,s}&\stackrel{\eqn{Cinv1}}{=}&\frac{-2}{2s-1}\sum_{n=0}^{k-2s}{2\choose k-2s-n}{n+2s-2\choose n}B_n\\
&=&\frac{-2}{2s-1}{2\choose 1}{k-2s-1+2s-2\choose k-2s-1}B_{k-2s-1}\\
&=&\frac{-4}{2s-1}{2K-2\choose 2s-2}B_{2K-2s}\\
&=&-4{2K-2\choose 2s-1}\frac{B_{2K-2s}}{2K-2s}.
\end{eqnarray*}
\end{proof}

Corollary \corol{r2} is just the restatement of the fact $\cC\cC^{-1}=I_{K-1}$.

\section{Proof of Theorem \thm{FFR}}\label{sc3}
In this section, we will state four lemmas, and use them to prove Theorem \thm{FFR}. The proofs of those lemmas will be provided in the next section.

The first lemma tells us that the second derivatives in $y$ of $F_k^{(i)}(x,y)$ and $G_k^{(i)}(x,y)$ can be expressed as linear combinations of $F_{k-2}^{(j)}(x,y)$ and $G_{k-2}^{(j)}(x,y)$ for some $j$'s.
\begin{lemma}\label{lem:dyF}
Let $k=2K+1\geq 5$ be an odd integer, for any $0\leq i\leq k-2$, we have
\begin{eqnarray}
\frac{\partial^2}{\partial y^2}\left(F_k^{(i)}(x,y)\right)&=&(k-2-i)(k-3-i)F_{k-2}^{(i)}(x,y) \label{eq:dyF}\\ 
&-&2i(k-2-i)G_{k-2}^{(k-3-i)}(x,y) \nonumber\\
&+&i(i-1)F_{k-2}^{(i-2)}(x,y) \nonumber\\
\frac{\partial^2}{\partial y^2}\left(G_k^{(i)}(x,y)\right)&=&(k-2-i)(k-3-i)G_{k-2}^{(i)}(x,y) \label{eq:dyG}\\ 
&-&2i(k-2-i)F_{k-2}^{(k-3-i)}(x,y) \nonumber\\
&+&i(i-1)G_{k-2}^{(i-2)}(x,y) \nonumber
\end{eqnarray}
\end{lemma}
Although we do not have the definitions of $F_{k-2}^{(i)}$ and $G_{k-2}^{(i)}$ for $i=-1$ or $-2$, we can see from the above expressions that the coefficients before them would be $0$.

The next one says that $F_k^{(i)}(x,y)$ and $G_k^{(i)}(x,y)$ are actually identical to each other.
\begin{lemma}\label{lem:FeG}
Let $k=2K+1\geq 5$ be an odd integer, for any $0\leq i\leq k-2$, we have
\begin{equation}
F_k^{(i)}(x,y)=G_k^{(i)}(x,y)
\end{equation}
\end{lemma}
Later, in Section \ref{sc5}, we will show that the above lemma can be extended to a more general setting, which will give us more Bernoulli number identities. Not only $F_k^{(i)}(x,y)$ and $G_k^{(i)}(x,y)$ have nice properties for the second derivatives in $y$, $R_k^{(i)}$ also has one as the next lemma claims.
\begin{lemma}\label{lem:dyR}
Let $k=2K+1\geq 5$ be an odd integer, for any $0\leq i\leq k-2$
\begin{eqnarray}
\frac{\partial^2}{\partial y^2}\left(R_k^{(i)}(x-y,y)\right)&=&(k-2-i)(k-3-i)R_{k-2}^{(i)}(x-y,y)\\ 
&-&2i(k-2-i)R_{k-2}^{(k-3-i)}(x-y,y) \nonumber\\
&+&i(i-1)R_{k-2}^{(i-2)}(x-y,y) \nonumber
\end{eqnarray}
\end{lemma}
Similarly, it is fine that we do not have the definitions for $R_{k-2}^{(-1)}(x,y)$ and $R_{k-2}^{(-2)}(x,y)$ in the above lemma.

The next lemma gives the tool to compute the coefficient of $x^{k-3}y^1$ in $F_k^{(i)}(x,x-y)$.
\begin{lemma}\label{lem:Ae0}
For any positive integers $k$ and $i$, we have
\begin{equation}
A_k^{(i)}:=\sum_{s=1}^{[\frac{k}{2}]}\frac{-2}{2s-1}\sum_{n=0}^{k-2s}{i\choose k-2s-n}{n+2s-2\choose n}B_n=0,
\end{equation}
where $B_n$ denotes the $n$th Bernoulli number.
\end{lemma}
Now we can use the above four lemmas to prove Theorem \thm{FFR}.
\begin{proof}[Proof of Theorem \thm{FFR}]
Let $k=2K+1\geq 5$ be an odd integer, and $i$ be any integer satisfying $0\leq i\leq k-2$. We only need to prove the following identity, which is obtained from \eqn{FFR} by change-of-variables $x\to x-y$ and $y\to y$
\begin{equation}\label{eq:FFR2}
F_k^{(i)}(x,x-y)+F_k^{(i)}(x,y)=R_k^{(i)}(x-y,y).
\end{equation}
When $i=0$ or $i=k-2$, both sides equal to $0$ by the definition. Hence we only need to prove \eqn{FFR2} for $1\leq i\leq k-3$. We will prove the statement by induction. For $k=5$, it is easy to check that 
\begin{eqnarray*}
F_5^{(1)}(x,x-y)+F_5^{(1)}(x,y)=R_5^{(1)}(x-y,y)&=&-\frac{1}{3}x^3;\\
F_5^{(2)}(x,x-y)+F_5^{(2)}(x,y)=R_5^{(2)}(x-y,y)&=&\frac{1}{3}x^3.
\end{eqnarray*}
Assume that 
\begin{equation*}
F_{k-2}^{(i)}(x,x-y)+F_{k-2}^{(i)}(x,y)=R_{k-2}^{(i)}(x-y,y).
\end{equation*}
is true for any $i$ satisfying $0\leq i\leq k-4$.

For any $1\leq i\leq k-3$, by Lemma \lem{dyF}, Lemma \lem{FeG} and Lemma \lem{dyR}, we have
\begin{eqnarray*}
&&\frac{\partial^2}{\partial y^2}\bigg(F_k^{(i)}(x,x-y)+F_k^{(i)}(x,y)\bigg)\\
&\stackrel{Lemma\ \lem{dyF}}{=}&(k-2-i)(k-3-i)\bigg(F_{k-2}^{(i)}(x,x-y)+F_{k-2}^{(i)}(x,y)\bigg)\\ 
&-&2i(k-2-i)\bigg(G_{k-2}^{(k-3-i)}(x,x-y)+G_{k-2}^{(k-3-i)}(x,y)\bigg)\\
&+&i(i-1)\bigg(F_{k-2}^{(i-2)}(x,x-y)+F_{k-2}^{(i-2)}(x,y)\bigg)\\
&\stackrel{Lemma\ \lem{FeG}}{=}&(k-2-i)(k-3-i)\bigg(F_{k-2}^{(i)}(x,x-y)+F_{k-2}^{(i)}(x,y)\bigg)\\ 
&-&2i(k-2-i)\bigg(F_{k-2}^{(k-3-i)}(x,x-y)+F_{k-2}^{(k-3-i)}(x,y)\bigg)\\
&+&i(i-1)\bigg(F_{k-2}^{(i-2)}(x,x-y)+F_{k-2}^{(i-2)}(x,y)\bigg)\\
&\stackrel{Induction\ Hypothesis}{=}&(k-2-i)(k-3-i)R_{k-2}^{(i)}(x-y,y)\\ 
&-&2i(k-2-i)R_{k-2}^{(k-3-i)}(x-y,y)\\
&+&i(i-1)R_{k-2}^{(i-2)}(x-y,y)\\
&\stackrel{Lemma\ \lem{dyR}}{=}&\frac{\partial^2}{\partial y^2}\bigg(R_k^{(i)}(x-y,y)\bigg).
\end{eqnarray*}
Hence,
\begin{equation*}
F_k^{(i)}(x,x-y)+F_k^{(i)}(x,y)-R_k^{(i)}(x-y,y)=a_1x^{k-3}y^1+a_0x^{k-2}.
\end{equation*}
After making change-of-variable as $x\to x+y$ and $y\to y$, we get
\begin{equation*}
F_k^{(i)}(x+y,x)+F_k^{(i)}(x+y,y)-R_k^{(i)}(x,y)=a_1(x+y)^{k-3}y^1+a_0(x+y)^{k-2}.
\end{equation*}
Since both $F_k^{(i)}(x+y,x)+F_k^{(i)}(x+y,y)-R_k^{(i)}(x,y)$ and $a_0(x+y)^{k-2}$ are symmetric about $x$ and $y$, and $a_1(x+y)^{k-3}y^1$ is not, we have $a_1=0$, i.e.
\begin{equation*}
F_k^{(i)}(x,x-y)+F_k^{(i)}(x,y)-R_k^{(i)}(x-y,y)=a_0x^{k-2}.
\end{equation*}
Now let us consider the coefficients of $x^{k-2}$ in the LHS. By Lemma \lem{Ae0}, we have
\begin{eqnarray*}
&&\textrm{coefficient of $x^{k-2}$ in $F_k^{(i)}(x,x-y)+F_k^{(i)}(x,y)-R_k^{(i)}(x-y,y)$}\\
&=&\sum_{s=1}^{K-1}\left(\frac{-2}{2s-1}\sum^{k-2s}_{n=0}{i\choose k-2s-n}{n+2s-2\choose n}B_n\right)+0-\frac{k-2-2i}{k-2}\\
&=&A_k^{(i)}-\left(\frac{-2}{2K-1}\sum^{k-2K}_{n=0}{i\choose k-2K-n}{n+2K-2\choose n}B_n\right)-\frac{k-2-2i}{k-2}\\
&\stackrel{Lemma\ \lem{Ae0}}{=}&0-\left(\frac{-2}{k-2}\sum^{1}_{n=0}{i\choose 1-n}{n+k-3\choose n}B_n\right)-\frac{k-2-2i}{k-2}\\
&=&\frac{2}{k-2}\left(iB_0+(k-2)B_1\right)-\frac{k-2-2i}{k-2}\\
&=&0.
\end{eqnarray*}
Therefore, we have $a_0=0$, i.e.
\begin{equation*}
F_k^{(i)}(x,x-y)+F_k^{(i)}(x,y)=R_k^{(i)}(x-y,y).
\end{equation*}
By induction, we have proven the statement for $F$. The result for $G$ will follow directly from Lemma \lem{FeG}.
\end{proof}

\section{Proof of lemmas}\label{sc4}
In this section, we will prove all the lemmas stated in the last section.
\begin{proof}[Proof of Lemma \lem{dyF}]
Let $k=2K+1\geq 5$ be an odd integer, and $i$ be any integer satisfying $0\leq i\leq k-2$. 
When $i=0$ or $k-2$, both sides of \eqn{dyF} are $0$. Assume that we have $1\leq i\leq k-3$.
Let us compare two sides of \eqn{dyF}.
\begin{eqnarray*}
LHS&=&\frac{\partial^2}{\partial y^2}\left(F_k^{(i)}(x,y)\right)\\
&=&\sum_{s=2}^{K-1}\left(\frac{-2}{2s-1}\sum^{k-2s}_{n=0}{i\choose k-2s-n}{n+2s-2\choose n}B_n\right)(2s-1)(2s-2)x^{2K-2s}y^{2s-3}\\
RHS&=&(k-2-i)(k-3-i)F_{k-2}^{(i)}(x,y)-2i(k-2-i)G_{k-2}^{(k-3-i)}(x,y)+i(i-1)F_{k-2}^{(i-2)}(x,y)\\
&=&(k-2-i)(k-3-i)\sum_{s=1}^{K-2}\left(\frac{-2}{2s-1}\sum^{k-2-2s}_{n=0}{i\choose k-2-2s-n}{n+2s-2\choose n}B_n\right)x^{2K-2-2s}y^{2s-1}\\
&-&2i(k-2-i)\sum_{s=1}^{K-2}\left(\frac{2}{2s-1}\sum^{k-2-2s}_{n=0}{k-4-(k-3-i)\choose k-2-2s-n}{n+2s-2\choose n}B_n\right)x^{2K-2-2s}y^{2s-1}\\
&+&i(i-1)\sum_{s=1}^{K-2}\left(\frac{-2}{2s-1}\sum^{k-2-2s}_{n=0}{i-2\choose k-2-2s-n}{n+2s-2\choose n}B_n\right)x^{2K-2-2s}y^{2s-1}\\
&=&(k-2-i)(k-3-i)\sum_{s=2}^{K-1}\left(\frac{-2}{2s-3}\sum^{k-2s}_{n=0}{i\choose k-2s-n}{n+2s-4\choose n}B_n\right)x^{2K-2s}y^{2s-3}\\
&-&2i(k-2-i)\sum_{s=2}^{K-1}\left(\frac{2}{2s-3}\sum^{k-2s}_{n=0}{i-1\choose k-2s-n}{n+2s-4\choose n}B_n\right)x^{2K-2s}y^{2s-3}\\
&+&i(i-1)\sum_{s=2}^{K-1}\left(\frac{-2}{2s-3}\sum^{k-2s}_{n=0}{i-2\choose k-2s-n}{n+2s-4\choose n}B_n\right)x^{2K-2s}y^{2s-3}
\end{eqnarray*}
Fix any $s$ and $n$, the coefficient of $B_n x^{2K-2s}y^{2s-3}$ on the RHS will be
\begin{eqnarray*}
&&\frac{-2}{2s-3}{i\choose k-2s-n}{n+2s-4 \choose n}\\
&\times&\left[(k-2-i)(k-3-i)+2(i-k+2s+n)(k-2-i)+(i-k+2s+n)(i-1-k+2s+n)\right]\\
&=&\frac{-2}{2s-3}{i\choose k-2s-n}{n+2s-4 \choose n}(n+2s-2)(n+2s-3)\\
&=&-2(2s-2){i\choose k-2s-n}{n+2s-4\choose n}\frac{(n+2s-2)(n+2s-3)}{(2s-2)(2s-3)}\\
&=&-2(2s-2){i\choose k-2s-n}{n+2s-2\choose n},
\end{eqnarray*}
which is exactly the coefficient of $B_n x^{2K-2s}y^{2s-3}$ on the LHS. Therefore, we have shown \eqn{dyF}.
A similar computation proves \eqn{dyG}, so we will save the proof for length.
\end{proof}

In order to prove Lemma \lem{FeG}, we need the following lemma, which is known as Carlitz's symmetric Bernoulli number identity.
\begin{lemma}[\ref{Ca},\ref{Sh}]\label{lem:Ca}
For any positive integers $m,n$, we have
\begin{equation}
(-1)^m\sum_{k=0}^m{m\choose k}B_{n+k}=(-1)^n\sum_{k=0}^n {n\choose k}B_{m+k}
\end{equation}
\end{lemma}

\begin{proof}[Proof of Lemma \lem{FeG}]
Let $k=2K+1\geq 5$ be an odd integer, and $i$ be any integer satisfying $0\leq i\leq k-2$. We will use the induction on $k$ to prove the result. For $k=5$, it is easy to check from the definition that
\begin{eqnarray*}
F_5^{(0)}(x,y)=G_5^{(0)}(x,y)&=&0\\
F_5^{(1)}(x,y)=G_5^{(1)}(x,y)&=&-\frac{1}{3}x^2y\\
F_5^{(2)}(x,y)=G_5^{(2)}(x,y)&=&\frac{1}{3}x^2y\\
F_5^{(3)}(x,y)=G_5^{(3)}(x,y)&=&0.
\end{eqnarray*}
Assume that we have $F_{k-2}^{(i)}(x,y)=G_{k-2}^{(i)}(x,y)$ for any $i$ satisfying $0\leq i\leq k-4$. Then by Lemma \lem{dyF}, we have
\begin{eqnarray*}
\frac{\partial^2}{\partial y^2}\left(F_k^{(i)}(x,y)-G_k^{(i)}(x,y)\right)&\stackrel{Lemma\ \lem{dyF}}{=}&(k-2-i)(k-3-i)\big(F_{k-2}^{(i)}(x,y)-G_{k-2}^{(i)}(x,y)\big) \\ 
&-&2i(k-2-i)\big(G_{k-2}^{(k-3-i)}(x,y)-F_{k-2}^{(k-3-i)}(x,y)\big) \\
&+&i(i-1)\big(F_{k-2}^{(i-2)}(x,y)-G_{k-2}^{(i-2)}(x,y)\big) \\
&=&0.
\end{eqnarray*}
Hence, 
\begin{equation*}
F_k^{(i)}(x,y)-G_k^{(i)}(x,y)=a_1x^{k-3}y^1+a_0x^{k-2}.
\end{equation*}
According to the definition, there are no $x^{k-2}$-terms in both $F_k^{(i)}(x,y)$ and $G_k^{(i)}(x,y)$, i.e. we have
\begin{equation*}
F_k^{(i)}(x,y)-G_k^{(i)}(x,y)=a_1x^{k-3}y^1.
\end{equation*}
Now let us compare the $x^{k-3}y^1$-terms in $F_k^{(i)}(x,y)$ and $G_k^{(i)}(x,y)$.
\begin{eqnarray*}
&&\textrm{coefficient of $x^{k-3}y^1$ in $F_k^{(i)}(x,y)$}\\
&=&-2\sum^{k-2}_{n=0}{i\choose k-2-n}B_n\\
&=&-2\sum^{k-2}_{n=0}{i\choose i-k+2+n}B_n\\
&=&-2\sum^{i}_{j=0}{i\choose j}B_{j+(k-2-i)}\\
&\stackrel{Lemma\ \lem{Ca}}{=}&2\sum^{k-2-i}_{j=0}{k-2-i\choose j}B_{j+i}\\
&=&2\sum^{k-2}_{n=0}{k-2-i\choose n-i}B_n\\
&=&2\sum^{k-2}_{n=0}{k-2-i\choose k-2-n}B_n\\
&=&\textrm{coefficient of $x^{k-3}y^1$ in $G_k^{(i)}(x,y)$}.
\end{eqnarray*}
Therefore, we have $a_1=0$, i.e.
\begin{equation*}
F_k^{(i)}(x,y)-G_k^{(i)}(x,y)=0.
\end{equation*}
By induction, we have proven the lemma.
\end{proof}

\begin{proof}[Proof of Lemma \lem{dyR}]
This can be proven by a direct computation, and we will skip the proof here.
\end{proof}
\begin{proof}[Proof of Lemma \lem{Ae0}]
Let $k$ and $i$ be any positive integers. We can rewrite $A_k^{(i)}$ as follows
\begin{eqnarray*}
A_k^{(i)}&=&\sum_{s=1}^{[\frac{k}{2}]}\frac{-2}{2s-1}\sum_{n=0}^{k-2s}{i\choose k-2s-n}{n+2s-2\choose n}B_n\\
&=&\sum_{n=0}^{k-2}\bigg(\sum_{s=1}^{[\frac{k-n}{2}]}\frac{-2}{2s-1}{i\choose k-2s-n}{n+2s-2\choose n}\bigg)B_n.
\end{eqnarray*}
It is easy to check that $A_3^{i}=0$ for any $i$. Now we will use induction to prove the statement. Assume that we have $A_{k-1}^{(i)}=0$ for any $i$.
For any $i$ we have
\begin{eqnarray*}
A_k^{(i+1)}-A_k^{(i)}&=&\sum_{n=0}^{k-2}\bigg[\sum_{s=1}^{[\frac{k-n}{2}]}\frac{-2}{2s-1}\bigg({i+1\choose k-2s-n}-{i\choose k-2s-n}\bigg){n+2s-2\choose n}\bigg]B_n\\
&=&\sum_{n=0}^{k-2}\bigg[\sum_{s=1}^{[\frac{k-n}{2}]}\frac{-2}{2s-1}{i\choose k-1-2s-n}{n+2s-2\choose n}\bigg]B_n,
\end{eqnarray*}
When $n=k-2$, we have
$${i\choose k-1-2s-n}={i\choose 1-2s}=0\qquad\textrm{for any $s\geq 1$}.$$
Also when $s=\frac{k-n}{2}$ is an integer, we have
$${i\choose k-1-2s-n}={i\choose -1}=0.$$
Hence,
\begin{eqnarray*}
A_k^{(i+1)}-A_k^{(i)}&=&\sum_{n=0}^{k-2}\bigg[\sum_{s=1}^{[\frac{k-n}{2}]}\frac{-2}{2s-1}{i\choose k-1-2s-n}{n+2s-2\choose n}\bigg]B_n\\
&=&\sum_{n=0}^{k-3}\bigg[\sum_{s=1}^{[\frac{k-1-n}{2}]}\frac{-2}{2s-1}{i\choose k-1-2s-n}{n+2s-2\choose n}\bigg]B_n\\
&=&A_{k-1}^{(i)}\\
&=&0,
\end{eqnarray*}
i.e., we have $A_k^{(i+1)}=A_k^{(i)}$ for any $i$.
Now since $A_k^{(0)}=0$, we have $A_k^{(i)}$ for any $i$.
\end{proof}

\section{Extension of Definition \ref{dFG} and Lemma \lem{FeG}}\label{sc5}
In this section, we will give an extension of Definition \ref{dFG} and Lemma \lem{FeG} for arbitrary integer $i$. It will then give us more Bernoulli number identities.
\begin{definition} Let $k=2K+1\geq 5$ be an odd integer, for any integer $i$, we define the following two polynomials
\begin{eqnarray}
F_k^{(i)}(x,y)&:=&\sum_{s=1}^{K-1}\left(\frac{-2}{2s-1}\sum^{k-2s}_{n=0}{i\choose k-2s-n}{n+2s-2\choose n}B_n\right)x^{2K-2s}y^{2s-1};\\
G_k^{(i)}(x,y)&:=&\sum_{s=1}^{K-1}\left(\frac{2}{2s-1}\sum^{k-2s}_{n=0}{k-2-i\choose k-2s-n}{n+2s-2\choose n}B_n\right)x^{2K-2s}y^{2s-1},
\end{eqnarray}
where the binomial coefficient is defined by 
\begin{equation}
{x\choose y}=\frac{\Gamma(x+1)}{\Gamma(y+1)\Gamma(x-y+1)}.
\end{equation} 
\end{definition}
The binomial coefficient for negative arguments is explicitly computed by the following theorem.

\begin{theoremNoNum}[\cite{Kr}]
For negative integer $n$ and integer $k$, we have
\begin{equation}
{n\choose k}=\begin{cases}
\begin{displaystyle}(-1)^k{-n+k-1\choose k}\end{displaystyle}&\textrm{ if $k\geq 0$}\\
\begin{displaystyle}(-1)^{n-k}{-k-1\choose n-k}\end{displaystyle}&\textrm{ if $k\leq n$}\\
0&\textrm{ otherwise}
\end{cases}
\end{equation}
\end{theoremNoNum}

They satisfy the following properties of binomial coefficients \cite{Kr}.
\begin{eqnarray}
{x\choose y}&=&{x\choose x-y}\\
{x\choose y}{y\choose z}&=&{x\choose z}{x-z\choose y-z}\\
{x\choose y}&=&\frac{x}{y}{x-1\choose y-1}\qquad\textrm{(except $y=0$)}\\
{x\choose y}&=&{x-1\choose y}+{x-1\choose y-1}\qquad\textrm{(except $x=y=0$)}
\end{eqnarray}
In the proof of Lemma \lem{dyF} and Lemma \lem{FeG}, we only used the above properties of binomial coefficients along with Lemma \lem{Ca}, which itself is also proved using only the above properties. Therefore, we have
\begin{prop}\label{propo:FeG}
Let $k=2K+1\geq 5$ be an odd integer, for any integer $i$, we have
\begin{equation}
F_k^{(i)}(x,y)=G_k^{(i)}(x,y)
\end{equation}
\end{prop}

Comparing the coefficients on both sides, we get

\begin{cor}\label{cor:r3}
Let $k=2K+1\geq 5$ be an odd integer, for any integer $s$ satisfying $1\leq s\leq K$ and for any integer $i$, we have the following Bernoulli number identities
$$-\sum^{k-2s}_{n=0}{i\choose k-2s-n}{n+2s-2\choose n}B_n=\sum^{k-2s}_{n=0}{k-2-i\choose k-2s-n}{n+2s-2\choose n}B_n$$
\end{cor}

\section*{Acknowledgement}
This study was funded by NSF Grant No. DMS-1401122. The author would like to thank Romyar Sharifi for introducing him a project related to this area.

\bibliographystyle{amsplain}

\end{document}